\documentclass[12pt]{amsart}
\usepackage{amssymb,amscd,amsmath}
\usepackage[all]{xy}
\newtheorem{thm}{Theorem}[section]

\newtheorem{prop}[thm]{Proposition}
\newtheorem{lemma}[thm]{Lemma}

\theoremstyle{remark}

\theoremstyle{definition}
\newtheorem{defn}[thm]{Definition}

\renewcommand{\bar}{\overline}

\newcommand{\uG}{\underline{G}}
\newcommand{\uGsc}{\underline{G}^\mathrm{sc}}

\newcommand{\Irr}{\mathrm{Irr}\,}

\newcommand{\SO}{\mathrm{SO}}

\renewcommand{\O}{{\mathcal{O}}}

\newcommand{\C}{{\mathbb{C}}}

\newcommand{\F}{{\mathbb{F}}}

\newcommand{\tr}{\mathrm{tr}}

\newcommand{\ord}{\mathrm{ord}}
\newcommand{\Aut}{\mathrm{Aut}}

\newcommand{\GL}{\mathrm{GL}}
\newcommand{\SL}{\mathrm{SL}}
\newcommand{\SU}{\mathrm{SU}}
\newcommand{\PSL}{\mathrm{PSL}}

\newcommand{\Sp}{\mathrm{Sp}}

\newcommand{\Spec}{\mathrm{Spec}}
\newcommand{\Fact}{\mathrm{Fact\:}}

\newcommand{\U}{\mathcal U}
\renewcommand{\L}{\mathcal L}
\newcommand{\IL}{\mathcal IL}
\newcommand{\IU}{\mathcal IU}
\newcommand{\IO}{\mathcal IO}
\newcommand{\Alt}{{\raise 2pt\hbox{$\scriptstyle\bigwedge$}}}

\newcommand{\go}{\rightarrow}

\newcommand{\e}{\epsilon}
\newcommand{\x}{\chi}
\newcommand{\z}{\zeta}

\begin{document}
\title[On the Distribution of Values of Certain Word Maps]
{On the Distribution of Values of Certain Word Maps}

\author{Michael Larsen}
\email{mjlarsen@indiana.edu}
\address{Department of Mathematics\\
    Indiana University \\
    Bloomington, IN 47405\\
    U.S.A.}

\author{Aner Shalev}
\email{shalev@math.huji.ac.il}
\address{Einstein Institute of Mathematics\\
    Hebrew University \\
    Givat Ram, Jerusalem 91904\\
    Israel}

\thanks{Michael Larsen was partially supported by NSF Grant DMS-1101424.
Aner Shalev was partially supported by ERC Advanced Grant 247034. 
Both authors were partially supported 
by Bi-National Science Foundation United States-Israel Grant 2008194.}

\begin{abstract}
We prove that, for any positive integers $m, n$, the word map
$(x,y) \mapsto x^my^n$ is almost measure preserving on finite
simple groups. This extends the case $m=n=2$ obtained in 2009.
Along the way we obtain results of independent interest on fibers 
of word maps and on character values.
\end{abstract}

\maketitle

\section{Introduction}

Let $F_d$ be the free group of rank d, and let $w \in F_d$ be a word. 
Let $G$ be a finite group. Then $w$ induces
a word map $w_G: G^d \go G$, induced by substitution, 
which we sometimes denote by $w$. 
Word maps have been studied extensively in recent years.  See for instance \cite{LiSh1}, \cite{Sh2},
\cite{LaSh1}, \cite{LOST}, \cite{ScSh}, \cite{LaSh2}, \cite{LaShTi}, 
and the monograph \cite{Se} by Segal.

The word
map $w_G$ induces a probability distribution $P = P_{w,G}$ on $G$, where
for a subset $X \subseteq G$ we have $P(X) = |w_G^{-1}(X)|/|G|^d$.
Let $\mu = \mu_G$ be the uniform distribution on $G$. We shall be interested
in finding out how close $P$ is to the uniform distribution, in other words, how large 
the $L^1$ distance $\Vert P-\mu\Vert _1$ can be.

If $w$ is a primitive word, i.e., part of a free basis for $F_d$, then
$P_{w,G}$ is clearly uniform for every finite group $G$. Puder
and Parzanchevski proved in \cite{PP} that the converse holds too.
In fact if $P_{w,G}$ is uniform for infinitely many symmetric groups
$G = S_n$ then $w$ is primitive.

In this paper we are interested in cases where $P_{w,G}$ is 
{\it almost uniform}
as $G$ ranges over the family of finite (nonabelian) simple groups. 
This means that $\Vert P_{w,G} - \mu_G\Vert _1 \go 0$ as $G$ ranges over finite simple
groups and $|G| \go \infty$.  

The first result of this kind appears in \cite{GSh}, where it is
shown that the commutator map is almost uniform on finite simple groups.
In the same paper, it is also shown that the word map $x_1^2 x_2^2$ is almost
uniform on finite simple groups. Our main theorem generalizes this latter result
as follows.

\begin{thm} 
\label{main}
Let $m_1, m_2$ be positive integers and let
$w = x_1^{m_1} x_2^{m_2}$. Then $w$ is almost uniform on finite
simple groups.
\end{thm}

A particular case of this result, for $G = \PSL_2(q)$, was obtained
by Bandman and Garion in \cite{BG}. Theorem~\ref{main} for the case of
alternating groups $G = A_n$ follows from Theorem 1.18 in
\cite{LaSh1}. Hence it remains to deal with simple groups of Lie type.

Let us now outline the way Theorem~\ref{main} is proved.

In a large simple group $G$, there are many pairs of conjugacy classes 
$C_1,C_2$ such that choosing independent uniform random elements $y_1\in C_1$ and $y_2\in C_2$, 
the resulting probability distribution of $y_1y_2$ on $G$ is close to uniform.  Indeed, this is shown in \cite{Sh}. 
The basic method to prove that a particular pair of conjugacy
classes satisfies this property is to get good upper bounds on the size of 
irreducible character values at $C_1$ and $C_2$.
One can prove that, for any given $\epsilon > 0$, most elements $g \in G$ 
satisfy the property that 
$|\chi(g)| \le \chi(1)^\epsilon$ (see Proposition 4.2 below).
Setting $y_1 = x_1^{m_1}, y_2 = x_2^{m_2}$, one then needs to prove that the inverse images by
power maps of the set of ``bad'' elements of $G$ is also small.  The difficulty is that the
sizes of fibers of the $m$th power map vary enormously.  The technical heart of this paper is devoted to the problem
of controlling the size of ``typical'' fibers for groups of Lie type of high rank over small fields. Once this is achieved we prove Theorem~\ref{main} using
character methods.

In the final section of the paper, we show that another family of words is almost uniform on finite
simple groups. A word $w \in F_d$ is called {\it admissible} if every
variable $x_i$ which occurs in $w$ occurs once with exponent $1$ and once
with exponent $-1$. For example, the commutator word is admissible,
as well as many other words, such as $x_1 \cdots x_d x_1^{-1} \cdots x_d^{-1}$.

\begin{prop} 
\label{admissible}
Every admissible word $w \ne 1$ is almost uniform on finite
simple groups.
\end{prop}

Our original proof of this result relied on \cite{DN}, but here we
give a shorter proof based on the recent subsequent paper \cite{PS}
by Parzanchevski and Schul.

\section{Random Polynomials over Finite Fields}

In this section, we show, roughly, that a random degree $n$ polynomial over $\F_q$
rarely has more than $O(\log n)$ factors over $\F_q[x]$.
This is relevant because an element of $\GL_n(\F_q)$ with many $m$th roots must have a highly factorable characteristic polynomial.
Since we are not interested in $\GL_n(\F_q)$ itself but in special linear, unitary, orthogonal, and symplectic groups,
the precise statements must be modified accordingly.

We begin with a useful identity, which is essentially the Euler product formula for the zeta-function
of $\F_q[x]$.  Throughout the paper, we will denote by $P_n(x)$ the polynomial
$$\frac{\sum_{ij=n} \mu(i) x^j}{n},$$
where $\mu$ is the M{\"o}bius function.

\begin{lemma}
\label{identity}
If $q$ is a prime power, then $P_n(q)$ is a positive integer for all $n\ge 1$, and we have a power series identity
$$\prod_{n=1}^\infty (1-x^n)^{-P_n(q)} = \sum_{i=0}^\infty q^i x^i.$$
\end{lemma}

\begin{proof}
By a standard inclusion-exclusion argument, $P_n(q)$ is the number of
irreducible monic polynomials of degree $n$ over $\F_q$ and is therefore a positive integer.
Moreover,
\begin{align*}
\prod_{n=1}^\infty (1-x^n)^{-\frac{\sum_{ij=n}\mu(i)q^j}n}
&=\exp\Bigl(\sum _{n=1}^\infty \frac{\sum_{ij=n}\mu(i)q^j}n\sum_{k=1}^\infty \frac{x^{kn}}k\Bigr) \\
&=\exp\Bigl(\sum_{i,j,k}\frac{\mu(i)q^j x^{ijk}}{ijk}\Bigr)\\
&=\exp\Bigl(\sum_{j,m}\frac{q^j x^{jm}}{jm}\sum_{ik=m} \mu(i)\Bigr) \\
& = \exp\Bigl(\sum_{j}\frac{q^j x^{j}}{j}\Bigr) \\
& = (1-qx)^{-1}.
\end{align*}
\end{proof}

\begin{prop}
\label{euler}
Let $a_1\le a_2$ be positive and $q\ge 2$.
Let $e_1,e_2,\ldots$ denote an infinite sequence of non-negative integers such that 
\begin{equation}
\label{e-bound}
a_1 q^n/n \le e_n \le a_2 q^n/n
\end{equation}
for all $n\ge 1$, and let $d_0=1,d_1,\ldots$ be defined by
\begin{equation}
\label{def-c}
\sum_{n=0}^\infty d_n x^n := \prod_{n=1}^\infty (1-x^n)^{-e_n}.
\end{equation}
Then there exist $A_1,A_2$ depending only
on $a_1$ and $a_2$ respectively such that for all $n\ge 0$,
$$(n+1)^{A_1} q^n \le d_n \le (n+1)^{A_2} q^n.$$
\end{prop}

\begin{proof}
Lemma~\ref{identity} gives the proposition immediately for one particular sequence, namely
$e_n := P_n(q)$.  
By the binomial theorem, it gives the theorem more generally
for $e_n := \lambda P_n(q)$ for any constant $\lambda>0$.
Regarding the $d_i$ as functions in the $e_j$, each $d_i$ is nondecreasing in each $e_j$ separately, so the proposition holds as long as
$$a_1 P_n(q) \le e_n \le a_2 P_n(q)$$
for all $n$.  As 
$$\frac{\sum_{ij=n}\mu(i)q^j}{q^n}$$
is bounded away from $0$ and $\infty$ for all $q\ge 2$ and $n\ge 1$, the proposition holds.
\end{proof}

\begin{prop}
\label{y-power-bound}
Let $e_n$ be a sequence satisfying (\ref{e-bound}) for some $q$, $a_1$, and $a_2$.
Define $c_{i,j}$ by
$$\sum_{i,j} c_{i,j} x^i y^j := \prod_{i=1}^\infty (1- x^i y)^{-e_i}.$$
There exist $C$ and $D>1$ depending only on $a_1$ and $a_2$ such that for all
$0\le m\le n$,
$$\frac{\sum_{j=m+1}^n c_{n,j}}{\sum_{j=0}^n c_{n,j}} \le n^C D^{-m}.$$
\end{prop}

\begin{proof}
We have
$$(9/4)^m\sum_{j=m+1}^n c_{n,j} < \sum_{j=0}^n c_{n,j}(9/4)^j,$$
which is the $x^n$ coefficient of
$$\prod_{i=1}^\infty (1- (9/4)x^i)^{-e_i}.$$
This is less than or equal to the $x^n$ coefficient of
$$(1-(9/4)x)^{-e_1}\prod_{i=1}^\infty (1 - (3x/2)^i)^{-e_i}
= (1-(9/4)x)^{-e_1}\sum_{n=0}^\infty (3/2)^nd_n x^n.$$
By Proposition~\ref{euler}, there exists $A_2 > 0$ such that
$$d_n \le (n+1)^{A_2} q^n$$
so
\begin{align*}
\sum_{i=0}^n (9/4)^i \binom{i+e_1-1}{e_1-1} & (3/2)^{n-i}d_{n-i}  \\
&\le (3/2)^n\binom{n+e_1-1}{e_1-1} (n+1)^{A_2} q^n \sum_{i=0}^n (3/2q)^i  \\
&= O(n^{A_2+e_1-1}(3/2)^nq^n)
\end{align*}
since $q > 3/2$.  Thus,
$$\sum_{j=m+1}^n c_{n,j} = O(n^{A_2}(2q/3)^n),$$
which, together with the lower bound $n^{A_1} q^n$ for $\sum_{j=0}^n c_{n,j}$ given by
Proposition~\ref{euler}, implies the proposition.
\end{proof}

\begin{prop}
\label{SLn-prop}
Let $\L_n(q)$ denote the set of monic polynomials $P(x)\in\F_q[x]$ 
of degree $n$ such that $P(0)= (-1)^n$.
For all $k > 0$, there exists $N$ such that if $n\ge 2$ and
$q$ is a prime power, then the set of $P(x)\in \L_n(q)$ such that
$P(x)$ splits into more than than $N\log n$ irreducible factors
has at most $n^{-k}|\L_n(q)|$ elements. 
\end{prop}

\begin{proof}
Let $\IL_n(q)$ denote the set of irreducible monic polynomials in $\F_q[x]$, excluding $x$.
Thus $\IL_n(q)$ can be identified with the set of $n$-element $q$-Frobenius orbits
in $\bar\F_q^\times$, which means
$$|\IL_n(q)| = \frac{\sum_{ij=n}\mu(i) (q^j-1)}n.$$
Thus $|\IL_n(q)| = P_n(q)$ except for $n=1$, and $\IL_1(q) = q-1$.
The norm maps $N_{\F_{q^n}/\F_q}\colon \F_{q^n}^\times\to \F_q^\times$
are surjective, so they combine to give morphisms
$N_n\colon \IL_n(q)\to \F_q^\times = \GL_1(\F_q)$
for which every fiber has the same cardinality, namely $P_n(q)/(q-1)$,
except when $n=1$, when the cardinality is $1$.

If $c_{n,m}$ denotes the number of elements in $\L_n(q)$ with exactly $m$ irreducible factors, then
\begin{align*}
\sum_{m,n}c_{n,m}x^ny^m
&= \frac 1{q-1}\sum_\chi \prod_{n=1}^\infty \prod_{P\in \IL_n(q)}(1-\chi(N_n(P))x^n y)^{-1},\\
&= \frac 1{q-1}\sum_\chi\prod_{n=1}^\infty (1-x^{n\,\ord(\chi)}y^{\ord(\chi)})^{-|\IL_n(q)|/\ord(\chi)}
\end{align*}
where the sum ranges over all characters $\chi$ of $\GL_1(\F_q)$.  (The second equality holds 
because the composition of $\chi$ and $N_n$ gives a map from $\IL_n(q)$
to the cyclic group $\langle e^{2\pi i/\ord(\chi)}\rangle$ whose fibers all have the same cardinality.)
Let $c_{n,m}(\chi)$ denote the $x^n y^m$ coefficient of 
$$\prod_{n=1}^\infty (1-x^{n\,\ord(\chi)}y^{\ord(\chi)})^{-|\IL_n(q)|/\ord(\chi)}.$$
Note that $c_{n,m}(\chi)$ is real and non-negative for all $m$, $n$, and $\chi$.
Proposition~\ref{euler} implies that for all $\chi$ of order $\ge 2$,
$$\sum_{m=0}^{\infty} c_{m,n}(\chi) = O((n+1)^{A_2}q^{n/2})$$
for some absolute constant $A_2$.  Therefore, the number of elements of
$\L_n(q)$ with more than $m$ prime factors is 
$$\frac{\sum_{i=m+1}^n c_{n,i}(1)}{q-1} + O((n+1)^{A_2} q^{n/2-1}),$$
while
$$\sum_{i=0}^n c_{n,i}(1) = q^{n-1} + O((n+1)^{A_2} q^{n/2}).$$
By Proposition~\ref{y-power-bound},
$$\frac{\sum_{i=m+1}^n c_{n,i}(1)}{\sum_{i=0}^n c_{n,i}(1)} = n^C D^{-m}
= n^{C - \frac{m\log D}{\log n}},$$
where $C$ and $D>1$ are absolute constants.  It follows that if 
$$m > \frac{C+k}{\log D}\log n,$$
then
$$\frac{\sum_{i=m+1}^n c_{n,i}}{\sum_{i=0}^n c_{n,i}} < n^{-k}$$
\end{proof}

\begin{prop}
\label{SUn-prop}
Let $\U_n(q)$ denote the set of monic polynomials $P(x)\in \F_{q^2}[x]$ such that 
$$\bar P(x) = (-x)^n P(1/x),$$
where $\bar P$ denotes the polynomial obtained from $P$ by applying the $q$-Frobenius to each coefficient.  For all $k > 0$, there exists $N$ such that if $n\ge 2$ and
$q$ is a prime power, then the set of $P(x)\in \U_n(q)$ such that
$P(x)$ splits into more than than $N\log n$ irreducible factors over $\F_{q^2}[x]$
has at most $n^{-k}|\U_n(q)|$ elements. 
\end{prop}

\begin{proof}
The elements of $\U_n(q)$ are in bijective correspondence with $n$-element multisets in
$\bar\F_q$ with product $1$ which are stable under the map $T_q\colon x\mapsto x^{-q}$.
Let $\IU_n(q)$ denote the set of minimal $n$-element $T_q$-stable 
subsets of $\bar\F_q^\times$.
When $n$ is even, we can identify $\IU_n(q)$
with ($n$-element) $T_q$-orbits of elements $\alpha\in \F_{q^n}^\times$ such that 
$\F_q(\alpha) = \F_{q^n}$.  When $n$ is odd, we can identify it
with $T_q$-orbits of elements $\alpha\in \ker \F_{q^{2n}}^\times\to \F_{q^n}^\times$
such that $\F_q(\alpha) = \F_{q^{2n}}$.  Either way, the product of the elements in the
orbit is $\alpha^{1-q+q^2 - \cdots + (-q)^{n-1}}\in U_1(\F_q)$.
The $1-q+\cdots+(-q)^{n-1}$ power map gives a surjection from $\F_{q^n}^\times$ to $U_1(\F_q)$ when $n$ is even and a surjection from $U_1(\F_{q^n})$
to $U_1(\F_q)$ when $n$ is odd.  Thus, 
$$|\IU_n(q)| = \frac {\sum_{d\mid n} \mu(n/d) (q^d - (-1)^d)}{n};$$
this coincides with $P_n(q)$ when $n>2$ but equals $P_1(q)+1$ for $n=1$ and $P_2(q)-1$
for $n=2$.
Moreover, every fiber of the map from $\IU_n(q)$  to $\ker\F_{q^2}^\times \to \F_q^\times$ has the same cardinality.  The argument now goes through exactly as for $\L_n(q)$.
\end{proof}

\begin{prop}
\label{SOn-prop}
Let $\O_{2n}(q)$ denote the set of monic polynomials $P(x)\in \F_q[x]$ of degree $2n$ such that
$P(x) = x^{2n} P(1/x)$.
For all $k > 0$, there exists $N$ such that if $n\ge 2$ and
$q$ is a prime power, then the set of $P(x)\in \O_{2n}(q)$ such that
$P(x)$ splits into more than than $N\log n$ irreducible factors in $\F_q[x]$
has at most $n^{-k}|\O_{2n}(q)|$ elements. 
\end{prop}

\begin{proof}
The elements of $\O_{2n}(q)$ are in bijective
correspondence with Frobenius-stable multisets in $\bar\F_q$ of cardinality $2n$
which are invariant under $x\mapsto x^{-1}$ and have product $1$.  

Let $\IO_m(q)$ denote the set
of  subsets of $\bar\F_q$ of cardinality $m$
which are stable under Frobenius and the map $x\mapsto x^{-1}$ and which are minimal among sets having this stability.
The elements of $\IO_m(q)$ with one element are $\{1\}$ and $\{-1\}$, which are distinct
if and only if $q$ is odd. 
Suppose $S\in \IO_m(q)$ has at least two elements, and let $\alpha\in S$.  
Now suppose $\alpha^{-1} = \alpha^{q^k}$ for some $k>0$.  We choose $k$ to be minimal, so 
$$S = \{\alpha,\alpha^q,\ldots,\alpha^{q^k} = \alpha^{-1},\alpha^{-q},\ldots,\alpha^{-q^{k-1}}\},$$
and $m = 2k$.  If no such $k$ exists, then,
or 
$$S = \{\alpha,\alpha^q,\ldots,\alpha^{q^{k-1}}\} 
\coprod \{\alpha^{-1},\alpha^{-q},\ldots,\alpha^{-q^{k-1}}\},$$
and again $m=2k$.  We conclude that
$$|\IO_m(q)| 
= \begin{cases}
2&\text{if $m=1$ and $q$ is odd,}\\
1&\text{if $m=1$ and $q$ is even,}\\
q-2&\text{if $m=2$ and $q$ is odd.}\\
q-1&\text{if $m=2$ and $q$ is even.}\\
0&\text{if $m\ge 3$ is odd,}\\
P_n(q) &\text{if $m=2n\ge 4$.}\\
\end{cases}$$

Let $\IO^+_{2n}(q)=\IO_{2n}(q)$ if $n\ge 2$ and
$$\IO^+_2(q) = \IO_{2}(q)\cup\{\{1,1\},\{-1,-1\}\},$$
where $\{x,x\}$ denotes the multiset whose unique element $x$ has multiplicity $2$.
In particular, $|\IO_{2n}(q)| = P_n(q)$
for all $n\ge 1$, and
every element in $\O_{2n}(q)$ decomposes uniquely as a sum of elements in
$\IO^+_{2n_i}(q)$ with $\sum n_i = n$.  
If $c_{n,m}$ denotes the number of elements in $\O_{2n}(q)$ whose root multiset
decomposes into exactly $m$ elements of $\IO^+_{2n_i}$, then
$$\sum_{m,n} c_{n,m}x^n y^m = \prod_{n=1}^\infty (1-x^n y)^{-P_n(q)},$$
so
$$\frac{\sum_{i=m+1}^n c_{n,i}}{\sum_{i=0}^n c_{n,i}} < n^C D^{-m}$$
where $C$ and $D>1$ are absolute constants, and it follows that the proportion of
such elements is $<n^{-k}$ if $N$ is sufficiently large.  Finally, if $P(x)$ has at least $m$
irreducible factors, its root multiset must decompose into at least $m/2$
elements of $\IO^+_{2n_i}(q)$, and the proposition follows.

\end{proof}

\section{Blocks of Typical Elements in Classical Groups}

\begin{defn}
A \emph{classical group} $\uG$ over a finite field is a group of type 
$\SL_n$, $\SU_n$, $\SO_{2n}^{\pm}$, $\Sp_{2n}$, or $\SO_{2n+1}$.
\end{defn}

The classical groups admit natural injective representations of degree
$n$, $n$, $2n$, $2n$, $2n$, and $2n+1$ respectively.  We identify
elements of classical groups with their images under these natural representations,
so that, e.g., we can speak of the eigenvalues of an element
$g\in \uG(\F_q)$.  These eigenvalues form a multiset $\Spec\:g$ which is invariant
under $x\mapsto x^q$ in the first case, $x\mapsto x^{-q}$ in the second case, and
both in the remaining four cases.  In the last case, $1$ is an eigenvalue of odd multiplicity,
and we define $\Spec_0 g$ to be the spectrum less one copy of $1$.  
For each $g$,  the characteristic polynomial of $g$ (divided by $x-1$ in the odd-orthogonal case) admits a unique decomposition into \emph{factors}, i.e., 
elements of $\IL_*(q)$ for $\uG = \SL_n$, $\IU_*(q)$ for $\uG= \SU_n$, or $\IO_{2*}(q)$ for $\uG$ self-dual
which corresponds to the decomposition of $\Spec\:g$ (or $\Spec_0 g$ in the odd-orthogonal case) into \emph{orbits}.
The total number of factors will be denoted $\Fact g$.

\begin{prop}
\label{centralizer}
There exist absolute constants $A>0$ and $B$ such that for every finite field $\F_q$,
every classical group $\uG/\F_q$ and
semisimple element $g\in \uG(\F_q)$, we have
$$|C_{\uG(\F_q)}(g)| > A (\Fact g)^B q^{\dim C_{\uG}(g)}.$$
\end{prop}

\begin{proof}
Consider first $\uG = \SL_n$.  Let $c_{i,j}$ denote the number of Galois orbits of eigenvalues
of $g$ of orbit size $i$ and multiplicity $j$.  Setting $C_i := \sum_j c_{i,j}$, we have
\begin{align*}|C_{\uG(\F_q)}(g)| &= \frac{\prod_{i,j}|\GL_j(\F_{q^i})|^{c_{i,j}}}{q-1}
> q^{\dim C_{\uG}(g)}\prod_{i,j} \prod_{k=1}^\infty (1-q^{-ki})^{c_{i,j}} \\
& = q^{\dim C_{\uG}(g)}\prod_i \prod_{k=1}^\infty (1-q^{-ki})^{C_i}.\\
\end{align*}
for $0<x<1/2$, 
$$\log(1-x) = -\sum_{k=1}^\infty x^k / k >  -\sum_{k=1}^\infty x^k > -2x,$$
so
$$\prod_{k=1}^\infty (1-q^{-ki}) > e^{-2q^{-i}-2q^{-2i}-\cdots} > e^{-4q^{-i}}.$$
Thus,
$$|C_{\uG(\F_q)}(g)| > q^{\dim C_{\uG}(g)} e^{-4\sum C_i q^{-i}}.$$
As $C_i < q^i$ for each $i$ and $\sum_i C_i = \Fact g$, we have
$$\sum_{i=1}^\infty C_i q^{-i} <1 + \lfloor \log_q \Fact g\rfloor < 1+\frac{\log \Fact g}{\log 2}.$$
Thus,
$$e^{-4\sum C_i q^{-i}} > e^{-4} (\Fact g)^{4/\log 2},$$
which gives the proposition in this case.
The other classical groups can be treated in the same way.

\end{proof}

\begin{prop}
\label{ss-cc}
There exist absolute constants $A$ and $B$ such that for every finite field $\F_q$,
every classical group $\uG/\F_q$ of dimension $d$ and rank $r$, and every
semisimple element $g\in \uG(\F_q)$, the set of elements in $\uG(\F_q)$
whose semisimple part is conjugate to $g$ has cardinality at most
$$A (\Fact g)^B q^{d-r}.$$
\end{prop}

\begin{proof}
To count elements with Jordan decomposition $su$ such that $s$ is conjugate to $g$,
we first count conjugates $s$ of $g$, and then count the number of unipotent elements in
each $C_G(s)$ (which does not depend on the choice of conjugate $s$).
By a standard estimate, 
$$ c_1 q^{\dim \uG} < |\uG(\F_q)| < c_2 q^{\dim \uG}$$
for some non-zero constants $c_1$ and $c_2$ which do not depend on $\uG$.
By Proposition~\ref{centralizer},
$$|g^{\uG(\F_q)}| = \frac{|\uG(\F_q)|}{|C_{\uG(\F_q)}(g)|} 
< c_3  (\Fact g)^{-B} q^{d-\dim C_{\uG}(g)}.$$
Every unipotent element of $C_{\uG(\F_q)}(g)$ lies in the identity component of
the reductive algebraic group $C_{\uG}(g)$, so
by a theorem of Steinberg \cite[Corollary 15.3]{St}, the number of unipotent elements in $C_{\uG(\F_q)}(g)$
is exactly $q^{\dim C_{\uG}(g) - r}$.  The proposition follows.
\end{proof}

\begin{prop}
\label{few-factors}
For all $\epsilon>0$, there exists $k$ such that for every finite field $\F_q$ and classical group $\uG/\F_q$, there exists a subset $S\subset \uG(\F_q)$ with
$|S|\le \epsilon |\uG(\F_q)|$, such that for all $g\not \in S$, $\Fact g < k\log n$.
\end{prop}

\begin{proof}
This follows immediately from Proposition~\ref{ss-cc}, the trivial bound $\Fact g \le n$,
and the bounds given by Propositions~\ref{SLn-prop}, \ref{SUn-prop}, 
and \ref{SOn-prop}
in the cases $\uG = \SL_n$, $\uG = \SU_n$, and $\uG$ is self-dual, respectively.
\end{proof}

\begin{thm}
\label{power-fibers}
For all positive integers $m$ and every $\epsilon > 0$, there exists $l$ such that for every classical group $\uG/\F_q$
there exists $S\subset \uG(\F_q)$ with
$|S|\le \epsilon |\uG(\F_q)|$, such that for all $g\not\in S$, 
$$|\{h\in \uG(\F_q)\mid h^m = g^m\}| < q^{l \log^4 n}.$$
\end{thm}

\begin{proof}
Whether we are in the $\IL$, the $\IU$, or the $\IO$ setting,
if $X\subset \bar\F_q^\times$ is an orbit such that
$\alpha^m=\beta^m$ implies $\alpha=\beta$ for all $\alpha,\beta\in X$, 
then 
$$X_m := X \coprod \zeta_m X \coprod\cdots\coprod \zeta_m^{m-1} X$$
decomposes into a disjoint union of
at most $m$ orbits, and each orbit has cardinality $\ge |X|$.  Indeed, each such orbit $Y$ is contained in $X_m$ and
maps onto $X$ by the map sending $\alpha\in Y$ to the unique $\beta\in X$
such that $\alpha^m=\beta^m$.
It follows that 
the number of elements $\IL_n(q)$, $\IU_n(q)$, or $\IO_n(q)$ respectively
with a root in the set $X_m$ is at most $O(mq^{n-|X|})$.

Moreover, whether we are in the $\IL$, the $\IU$, or the $\IO$ setting, if $\alpha$ and $\zeta_m^i \alpha$
are distinct elements of the same orbit $X$ of size $k$, then without loss of generality we may assume that 
$$\zeta_m^i\alpha = \alpha^{(\pm q)^j}$$
for some $j\le k/2$.  The number of such elements for a given $i$ is $\le n(q^{k/2}+1)$,
so the number of elements in $\IL_n(q)$, $\IU_n(q)$, or $\IO_n(q)$ which have two distinct roots whose ratio is an $m$th root of unity is $O(mnq^{n/2})$.

By removing $\epsilon |\uG(\F_q)|$ elements from $\uG(\F_q)$ we can assume by Proposition~\ref{few-factors} 
that for all remaining elements $g$, we have $\Fact g = O(\log n)$.  
By omitting all elements whose characteristic polynomial has a  factor of degree $\gg \log n$with two distinct roots whose ratio is 
an $m$th root of unity or two factors of degree $\gg \log n$  with roots whose ratio is an $m$th root of unity, 
we may further assume that there are at most $O(\log n)$ eigenvalues of $g^m$ with multiplicity $\ge 1$
and each such eigenvalue has multiplicity $O(\log n)$.

To bound the order of the set of $m$th roots of $g^m$ in $\uG(\F_q)$ we embed this group in a larger group $G$, which is 
defined to be $\GL_n(\F_q)$,
$\GL_n(\F_{q^2})$, $\GL_{2n}(\F_q)$, $\GL_{2n}(\F_q)$, or $\GL_{2n+1}(\F_q)$, depending on whether 
$\uG$ is linear, unitary, even-orthogonal, symplectic, or odd-orthogonal, and bound the number of $m$th
roots of $g^m$ in $G$.  Since $h$ commutes with $g^m=h^m$, all of the $m$th roots $h$ lie
in $C_G(g^m)$.   We can decompose the natural module on which $G$ acts ($\F_q^n$, $\F_{q^2}^n$, $\F_q^{2n}$, or $\F_q^{2n+1}$)
as a direct sum of  two subspaces, $V_1$ and  $V_2$, such that the action of $g^m$ respects this decomposition, $\dim V_2 = O(\log^2 n)$,
no eigenvalue of $g^m$ acting on $V_1$ coincides with any eigenvalue of $g^m$ acting on $V_2$, $g^m$ has regular semisimple action on $V_1$,
and the number of irreducible factors (over $\F := \F_q$ or $\F := \F_{q^2}$, as the case may be) of $g^m$ acting on $V_1$ is $O(\log n)$.
Thus, every $m$th root $h$ of $g^m$ respects the decomposition $V_1\oplus V_2$, and we can identify any such $h$ 
with a pair $(h_1,h_2)\in \Aut_{\F}(V_1)\times \Aut_{\F}(V_2)$.
The centralizer of $h_1^m$ in $\Aut_{\F}(V_1)$ is a product of $O(\log n)$ cyclic groups, so there are at most $m^{O(\log n)}$ possibilities for $h_1$.
There are at most $|\Aut_{\F}(V_2)| = q^{O(\log^4 n)}$ possibilities for $h_2$.  The theorem follows.

\end{proof}

\section{The Main Theorem}

In this section we prove Theorem~\ref{main}.
More precisely, we prove the following:

\begin{thm}
\label{main-precise}
Let $m_1$ and $m_2$ be fixed positive integers.  For a given group $G$, 
let $f\colon G^2\to G$
denote the map
$$f(x_1,x_2) = x_1^{m_1} x_2^{m_2}.$$
Let $\mu_G$ denote the uniform distribution on $G$ and $\mu_{G\times G}$ the
uniform distribution on $G^2$.  Then,
$$\lim_G \Vert \mu_G - f_*\mu_{G\times G}\Vert_1= 0$$
if the limit is taken over any sequence of pairwise distinct finite simple 
groups. 
\end{thm}

Here $f_*\mu_{G\times G}(X) :=  \mu_{G\times G}(f^{-1}(X))$ where
$X \subseteq G$. 

The proof uses, among other tools, the representation zeta function 
$\z^G$ of $G$, which we now define.

For a real number $s>0$ set 
\[
\z^G(s) = \sum_{\x \in \Irr G} \x(1)^{-s}.
\]

We need the following result, which is of independent interest.

\begin{prop} 
\label{generic}
Fix $\epsilon > 0$, and let $G$ be a finite simple group.
Then the probability that, for $g \in G$, $|\chi(g)| \le \chi(1)^{\e}$
for all irreducible characters $\chi$ of $G$ tends to $1$ as $|G| \go \infty$.
\end{prop}

\begin{proof}

By \cite{LiSh3} we have, for a fixed $s>0$, $\z^G(s) = 1+ o(1)$
unless $G$ is a finite simple group of Lie type of bounded rank.

By Lemma 2.2 of \cite{Sh}, the probability that 
$|\chi(g)| \le \chi(1)^{\e}$ for all irreducible characters $\chi$ of $G$
is at least $2-\z^G(2 \e)$. If $G$ is not of Lie type of bounded rank
then the latter expression is $1- o(1)$ as required.
If $G$ is of Lie type of bounded rank, then the probability that
$g \in G$ is regular semisimple is $1-o(1)$, and for regular 
semisimple elements $g$ there exists $N$ depending on the rank of
$G$ such that $|\x(g)| \le N$ for all $\x \in \Irr G$ \cite[Theorem~3]{GLL}.
If $G$ is large enough we have $N \le \chi(1)^{\e}$
for all $1 \ne \x \in \Irr G$, and the result follows.

\end{proof}

\begin{prop}
\label{bounded-rank}
For all positive integers $r, m$ and every $\epsilon > 0$, there exists $N$ such that for every group $G$ of Lie type of 
rank $r$ there exists a normal subset $S$ with $|S| \le \epsilon|G|$ such that  $|\chi(g^m)| \le N$ for every $g\in G\setminus S$
and every irreducible character of $G$.
\end{prop}

\begin{proof}
Note that this case includes the Suzuki and Ree groups, so we express $G$ as a quotient $\tilde G/Z(\tilde G)$,
where $\tilde G := \uG(\bar\F_q)^F$,
$\uG$ is a a simply connected  simple algebraic group, $F$ is a Frobenius map, and $Z$ denotes center.
Fixing the rank of  $\uG$ gives an upper bound for the order of the Weyl group of 
$\uG$ and the number of simple roots of $\uG$.  
For every regular semisimple element
$a\in G$, we choose a regular semisimple  element $\tilde a$ of  $\tilde G$ which lies over $a$.
To bound $|\chi(a)|$ for every irrreducible character of $G$, it suffices to bound $|\chi(\tilde a)|$ for every irreducible character of $\tilde G$.  Letting $\tilde A$ denote the centralizer of $\tilde a$ in $\tilde G$, we can deduce this by applying \cite[Theorem 3]{GLL} to the pair $(\tilde G, \tilde A)$.

If $\tilde G_m$ denotes the set of $x\in \tilde G$ such that $x^m$ fails to be regular semisimple, then
$\tilde G_m$ consists of the $F$-fixed points of a proper Zariski-closed subset of $\uG(\bar\F_q)$
since the $m$th power map is dominant and the regular semisimple locus is open and dense in $\uG$.
By the argument of \cite[Proposition~3.4]{LaSh2},
$$|\tilde  G_m| = O(|G|^{1-1/\dim \uG}).$$
By choosing $N$ sufficiently large, we may take $|G|$ as large as we wish and thereby obtain the proposition.

\end{proof}

We can now prove Theorem~\ref{main-precise}.
\begin{proof}
As this is an asymptotic statement, it suffices to consider only alternating groups and groups of Lie type.
For alternating groups, Theorem~\ref{main-precise} is a special case of \cite[Theorem~1.18]{LaSh1}. 

For any two conjugacy classes $C_1$ and $C_2$ of $G$ (possibly equal), we denote by $\mu_{C_1\times C_2}$ 
the uniform distribution on the set $C_1\times C_2\subset G^2$.  We can write
$$\mu_{G\times G} = \sum_{C_1\subset G}\sum_{C_2\subset G} \frac{|C_1\Vert C_2|}{|G|^2} \mu_{C_1\times C_2},$$
where the sums are taken over conjugacy classes.  Thus,
$$\Vert \mu_G - f_*\mu_{G\times G}\Vert_1 
  \le \sum _{C_1\subset G}\sum_{C_2\subset G} \frac{|C_1\Vert C_2|}{|G|^2} \Vert \mu_G - f_*\mu_{C_1\times C_2}\Vert_1.$$
If $S$ is a normal subset of $G$ such that $\frac {|S|}{|G|}  \le \frac \epsilon 5$ and 
$$\Vert \mu_G - f_*\mu_{C_1\times C_2}\Vert_1 \le \frac \epsilon 5$$
for all conjugacy classes $C_1,C_2\subset G\setminus S$, then 
\begin{align*}
\Vert \mu_G - f_*\mu_{G\times G}\Vert_1 
&\le \sum _{C_1\subset G\setminus S}\sum_{C_2\subset G\setminus S} \frac{|C_1\Vert C_2|}{|G|^2} \Vert \mu_G - f_*\mu_{C_1\times C_2}\Vert_1 \\
&\qquad\qquad  + \sum _{C_1\subset S}\sum_{C_2\subset G} \frac{|C_1\Vert C_2|}{|G|^2} \Vert \mu_G - f_*\mu_{C_1\times C_2}\Vert_1 \\
&\qquad\qquad  + \sum _{C_1\subset G}\sum_{C_2\subset S} \frac{|C_1\Vert C_2|}{|G|^2} \Vert \mu_G - f_*\mu_{C_1\times C_2}\Vert_1 \\
&\le \frac \epsilon5 + \frac{2\epsilon}5 + \frac{2\epsilon}5 = \epsilon.
\end{align*}
It therefore suffices to prove that as $|G|\to \infty$, we can choose a normal subset $S$ such that $\frac{|S|}{|G|}$
and   
\begin{equation}
\label{classes}
\max_{C_1,C_2\subset G\setminus S} \Vert \mu_G - f_* \mu_{C_1\times C_2}\Vert_1
\end{equation}
both tend to zero.   

If $c_1\in C_1$, $c_2\in C_2$, and $g\in G$, the probability that uniformly chosen
conjugates $x_1$ and $x_2$ of $c_1$ and $c_2$ respectively satisfy $x_1^{m_1} x_2^{m_2} = g$ is
$$\frac 1{|G|} \sum_\chi \frac {\chi(c_1^{m_1})\chi(c_2^{m_2})\bar\chi(g)}
{\chi(1)},$$
where $\chi$ ranges over all irreducible characters of $G$.  We claim that for the desired $L^1$ estimate, it suffices to prove that
there exists a normal set $S'$ of $G$ of cardinality $o(|G|)$, such that for all $g\in G\setminus S'$ and all $c_1,c_2\in G\setminus S$,
$$ \sum_{\chi\neq 1} \frac {\chi(c_1^{m_1})\chi(c_2^{m_2})\bar\chi(g)}{\chi(1)} = o(1).$$
Indeed, if $\mu_1$ denotes the measure obtained on $G$ by restricting
$f_* \mu_{C_1\times C_2}$ to $G\setminus S'$ and extending by zero, and 
$$\mu_2 = f_* \mu_{C_1\times C_2} - \mu_1,$$
then the $\mu_i$ are non-negative measures, so 
$$\Vert \mu_1\Vert_1 + \Vert \mu_2\Vert_1 = \Vert f_*\mu_{C_1\times C_2}\Vert_1 = 1.$$
Moreover,
$$\Vert  \mu_G -\mu_1\Vert_1 \le \frac{|S'|}{|G|} 
	+ \max_{g\in G\setminus S'} \frac 1{|G|} \sum_\chi \frac {\chi(c_1^{m_1})\chi(c_2^{m_2})\bar\chi(g)}{\chi(1)}.$$
As
\begin{equation*}
\begin{split}
\Vert \mu_G - f_*\mu_{C_1\times C_2}\Vert_1
	& \le \Vert  \mu_G -\mu_1\Vert_1 + \Vert \mu_2\Vert_1
	= \Vert  \mu_G -\mu_1\Vert_1  + 1 - \Vert \mu_1\Vert_1 \\
	& \le \Vert  \mu_G -\mu_1\Vert_1  + 1 - (\Vert \mu_G\Vert_1 - \Vert \mu_G - \mu_1\Vert_1) \\
	&= 2\Vert  \mu_G -\mu_1\Vert_1,
\end{split}
\end{equation*}
the claim follows.

We consider first the case that the rank $r$ of $\uG$ is sufficiently large.  In particular, this means that $G$ is 
of linear, unitary, orthogonal, or symplectic type, which means that there exists a classical group $\uG$ over a finite
field $\F_q$ and a subgroup $\tilde G\subset \uG(\F_q)$ of index $\le 2$ such that $\tilde G$ is a central extension of $G$ of
degree $\le r$.  By Theorem~\ref{power-fibers}, there exists a normal subset $\tilde S$ of $\tilde G$ of cardinality $o(|\tilde G|)$ 
such that for every $g\in \tilde G\setminus \tilde S$ the fibers of the $m_1$th power and $m_2$th power maps containing $g$ have cardinality
$q^{O(\log^4 r)}$.  If $\uGsc$ is the simply connected covering group of $\uG$, then $\uGsc(\F_q)\to G$ is a universal central extension and
therefore factors through $\tilde G$.  In particular, every irreducible representation of $\tilde G$ can be regarded as an irreducible representation of $\uGsc(\F_q)$.
Applying Proposition~\ref{generic}, we conclude that the (normal) set of all elements $g\in \tilde G\setminus \tilde S$ such that 
$$\max(|\chi(g)|, |\chi(g^{m_1})|,|\chi(g^{m_2})|) > \chi(1)^{1/4}$$
for some irreducible character $\chi$ of $\tilde G$ has cardinality $o(|\tilde G|)$.  By, \cite[Theorem~1.2]{LiSh3}, if $r$ is sufficiently large,
$$\sum_{\chi\neq 1} \chi(1)^{-1/4} = o(1).$$
Defining $\tilde f$ to be the map $(x_1,x_2)\mapsto x_1^{m_1} x_2^{m_2}$ on $\tilde G$,
$$\Vert \mu_{\tilde G} - \tilde f_*\mu_{\tilde G\times\tilde G}\Vert_1 = o(1).$$
As the diagram
$$\xymatrix{\tilde G\times \tilde G \ar[r]^(.6){\tilde f}\ar[d]&\tilde G\ar[d] \\
			G\times G\ar[r]^(.6)f& G}$$
commutes, we deduce
$$\Vert \mu_G - f_*\mu_{G\times G}\Vert_1 = o(1).$$

Next we consider the case that the rank $r$ is bounded but greater than $1$.  This includes the cases of
Suzuki and Ree groups.  Let $\tilde G = \uGsc(\bar\F_q)^F$ where $F$ is a Frobenius map.  By  Proposition~\ref{bounded-rank}, there exists a normal subset $\tilde S\subset \tilde G$ such that for all $g\in \tilde G$
and every irreducible character $\chi$ of $g$, 
\begin{equation}
\label{absolute-bound}
\max(|\chi(g)|, |\chi(g^{m_1})|,|\chi(g^{m_2})|) = O(1).
\end{equation}
If $S$ denotes the image of $\tilde S$ in $G$, it follows that $|S| = o(|G|)$ (since the degree of the morphism $\uGsc\to \uG$ is bounded by rank), and that
(\ref{absolute-bound}) holds for $g\in G\setminus S$ and $\chi$ any character of $G$.
By Theorem~1.1 and Corollary~1.3 of \cite{LiSh3}, 
$$\sum_{\chi\neq 1} \frac 1{\chi(1)} = \z^G(1) - 1 = o(1)$$
for all finite $G$ not of rank $1$.  This implies the theorem.

Finally, we consider the case of groups of the form $G=\PSL_2(q)$.  It suffices to prove the theorem for $\SL_2(\F_q)$.
As there are $O(m)$ regular semisimple classes in $\SL_2(\F_q)$ whose $m $th power fails to be regular semisimple,
it suffices to prove that (\ref{classes}) tends to zero if $G\setminus S$ contains only regular semisimple elements.

By elementary algebra one  checks that $x\in \SL_2(\F_q)$ is regular semisimple if and only if $\tr(x)\not \in\{-2,2\}$
and that if $a_1,a_2,a_3\in \F_q\setminus \{-2,2\}$ such that 
$$a_1^2+a_2^2+a_3^2 \neq a_1 a_2 a_3 + 4,$$
the diagonal conjuation action of $\SL_2(\F_q)$ on
$$\{(x_1,x_2,x_3)\in \SL_2(\F_q)^3\colon x_1x_2x_3=e,\: \tr(x_i) = a_i\;\forall i\in\{1,2,3\}\}$$
is simply transitive.
In particular, the number of ways to write $x_3^{-1}$
as a product of a conjugate of $x_1$ and a conjugate of $x_2$ equals the order of the centralizer of $x_3$ in $\SL_2(\F_q)$,
which is either $q-1$ or $q+1$ depending on whether $x_3$ belongs to a split or a non-split torus of $\SL_2$ over $\F_q$.
>From this, we deduce that (\ref{classes}) tends to zero as long as $G\setminus S$ contains only regular semisimple elements.

\end{proof}

\section{Admissible words}

We conclude this paper by proving Proposition~\ref{admissible}.
Let $w \in F_d$ ($d \ge 1$) be a word. 
For a finite group $G$ let $N_w(g)$ be the number of solutions of
the equation $w(g_1, \ldots , g_d) = g$ where $g_i \in G$.
Then $N_w$ is a class function on $G$, hence it can be expressed
as $N_w = \sum_{\x \in \Irr G} N_w^{\x} \cdot \x$, where
$N_w^{\x} \in \C$ are the so called Fourier coefficients,
which have been studied by many authors.

Now let $w$ be as in Proposition~\ref{admissible}, namely $1 \ne w \in F_d$ 
is admissible in the variables $x_1, \ldots , x_d$. Without
loss of generality we may assume that $x_1, \ldots , x_d$ 
all occur in $w$.

By formula (1.5) of \cite{PS} we have
\[
N_w^{\x} = |G|^{d-1}/\x(1)^{d-r},
\]
where $r$ is a
certain integer depending on $w$ and satisfying 
$1 \le r \le d+1$. If $r = d+1$ then $N_w^{\x} = |G|^{d-1} \x(1)$
for all $\x \in \Irr G$, so $N_w = |G|^{d-1} \sum_{\x} \x(1)\x$.
This implies that the word map $w:G^d \go G$ is identically $1$
so $w = 1$ in $F_d$, a contradiction. It follows that $r \le d$. 

Next, by Theorem 4.2 in \cite{PS}, $r$ is not congruent to $d$ modulo
$2$. Setting $k = d-r$ we obtain $k \ge 1$.

Clearly $P_{w,G} = N_w/|G|^d$, and this yields 
\[
P_{w,G} = \frac1{|G|} \sum_{\x \in \Irr G} \x(1)^{-k} \cdot \x.
\] 

By Lemma 2.3 of \cite{GSh}, if 
$P = |G|^{-1}\sum_{\x \in \Irr G} a_{\x}\x$ and $a_1 = 1$, then
$\Vert P-\mu_G\Vert _1 \le (\sum_{1 \ne \x \in \Irr G} |a_{\x}|^2)^{1/2}$.
Applying this we obtain 
\[
\Vert P_{w,G} - \mu_G\Vert _1 
\le  \bigl(\sum_{1 \ne \x \in \Irr G} \x(1)^{-2k}\bigr)^{1/2} 
= (\z^G(2k)-1)^{1/2}. 
\]
By \cite{LiSh2} we have $\z^G(2k) \go 1$ as $|G| \go \infty$.
Therefore $\Vert P_{w,G} - \mu_G\Vert _1 \go 0$, and 
this completes the proof of Proposition~\ref{admissible}.

\end{document}